\documentclass{pspum-l}
\usepackage{cite}
\usepackage{amsrefs, amsmath, palatino, mathpazo, amsfonts, amssymb, mathtools}
\usepackage[mathscr]{eucal}
\usepackage{tikz}
\usepackage{tikz-cd}
\usepackage[all]{xy}
\usepackage{datetime}

\newtheorem{theorem}{Theorem}[section]

\newtheorem{proposition}[theorem]{Proposition}
\newtheorem{corollary}[theorem]{Corollary}

\newtheorem*{theorem*}{Theorem}
\newtheorem*{lemma*}{Lemma}
\newtheorem*{proposition*}{Proposition}
\newtheorem*{corollary*}{Corollary}
\newtheorem*{definition*}{Definition}
\theoremstyle{remark}
\newtheorem{remark}[theorem]{Remark}
\newtheorem{example}[theorem]{Example}

\newtheorem*{conventions*}{Conventions}
\newtheorem*{example*}{Example}

\numberwithin{equation}{section}

\newcommand{\Mbar}{\overline{M}}
\newcommand{\p}{\partial}

\newcommand{\op}{\operatorname}
\newcommand{\so}{\mathscr{O}}

\newcommand{\C}{\mathbb{C}}
\newcommand{\sO}{\mathscr{O}}
\newcommand{\Pp}{P}
\newcommand{\Bl}{\operatorname{Bl}}
\newcommand{\fX}{\mathfrak{X}}
\newcommand{\NE}{\operatorname{NE}}

\title{Variations on the theme of quantum Lefschetz}

\author[H.~Fan]{Honglu Fan}
\address{D-Math, ETH Z\"urich, R\"amistrasse 101, 8092, Z\"urich, Switzerland}
\email{honglu.fan@math.ethz.ch}

\author[Y.-P.~Lee]{Yuan-Pin~Lee}
\address{Department of Mathematics, University of Utah,
Salt Lake City, Utah 84112-0090, U.S.A.}
\email{yplee@math.utah.edu}

\date{\today}

\begin{document}

\maketitle

\begin{abstract}
In this companion piece to \cite{FL2}, some variations on the main results there are sketched.
In particular, 
\begin{itemize}
\item the recursions in \cite{FL2}, which we interpreted as the quantum Lefschetz, is reformulated in terms of Givental's quantization formalism, or equivalently, a summation of \emph{finitely many} graphs;
\item varieties of modification of the auxilliary spaces (masterspaces) for the fixed point localization are given, leading to different (looking) recursions;
\item applications of this circle of ideas to derive (apparently) new relations of Gromov--Witten invariants. 
\end{itemize}
\end{abstract}

\setcounter{section}{-1}

\section{Introduction}
Establishing a \emph{quantum Lefschetz hyperplane theorem} (QLHT) is an important problem in the Gromov--Witten theory.
From the enumerative point of view, one often encounters enumerative problems on hypersurfaces or complete interesections whose ambient spaces have much easier corresponding enumerative problems.
The celebrated quintic threefold in $P^4$ is one such example, where the Gromov--Witten invariants of $P^4$ have been completely determined by A.~Givental in \cite{aGhg}.
From the functoriality point of view, the QLHT and the quantum Leray--Hirsch are the two pillars in the study of the functoriality of Gromov--Witten theory \cites{LLW-string, LLW-ams}.
The quantum Leray-Hirsch for toric fibration was esatablished in genus zero \cites{GB, LLW2}, and in higher genus \cite{CGT}.
However, the quantum Lefschetz for higher genus was only available in very limited form.
For Calabi--Yau threefolds, there have been various approaches, most of which put a special focus on the quintic hypersurface of $\Pp^4$.  In genus $1$, there are works in \cites{Z,LK} among others. In higher genus there is an ongoing \emph{mixed spin $P$-field} approach pioneered by H.-L.~Chang, S.~Guo, J.~Li, W.-P.~Li and C.-C.~Liu in, e.g., \cites{CLLL, CGLZ, CGL, CGL2}. S.~Guo, F.~Janda and Y.~Ruan in \cites{GJR, GJR2} proposed a new method which determines higher genus invariants of quintic $3$-folds via twisted invariants of the ambient space $\Pp^4$ plus some ``effective invariants'' which are finite for a fixed genus. In the same paper the authors used this in genus $2$ to prove the BCOV holomorphic anomaly conjecture \cite{BCOV2}, and we expect it to be generalized to higher genus. 
However, beyond the Calabi--Yau threefolds, there was no proposal until our previous work.

The goal of this paper is to amplify a few points in \cite{FL2}. Therefore, this is a companion piece and will tacitly assume that the readers are familiar with the previous paper.
There we derive a recursion relation which computes the GW invariants of an ample hypersurface in terms of (twisted) GW invariants of the ambient space and a finite number of ``initial data'' for each genus.
The initial data are finite number of GW invariants on the hypersurface of low degrees.
In this paper, we will reinterpret the recursion relation in the frame work of Givental's quantization formalism \cites{aGhg, aGqh}.
By way of the quantization formalism, the recusion can be written as a summation of finite hypergraphs for each fixed genus.
This interpretation simplifies when the hypersurface in question is a Calabi--Yau threefold, but it applies to the general case.
Another feature of this approach is its flexibility of allowing primitive class insertions in computing GW invariants of the hypersurfaces, cf.\ Remark~\ref{r:prim}.
This is sketched in Section~1.

The second point is to explore varieties of related but different recursions using the same technique: localization on the masterspace.
Utilizing different masterspaces, we reached different recursions. 
In Section~2.1, the GW theory of the hypersurfaces are related to different twisted theories of the ambient spaces.
In Section~2.2, the original masterspace for the quintic is, by way of $K$-equivalence (crepant transformation), modified to an orbifold masterspace, where one of the fixed point becomes an orbifold point $B \mu_5$.
The GW theory of quintic is thus related orbifold theory of $B \mu_5$ recursively.

In Section~3, we flip the coin and consider the case when the hypersurface has trivial GW theory.
The aformentioned relation thus becomes a relation of GW invariants of the ambient space.
This works, for example, if the ambient space is a Fano threefold and the hypersurface is a K3 suface, representing the canonical divisor.
We wrote down some sample equestions in this case and hope to come back to applications of these types of equations in the future.

\subsection*{Acknowledgement}
We wish to thank T.~Jarvis, N.~Priddis and Y.~Ruan, the organizers of the Snowbird Conference and School, for the invitations to attend the conference and to submit the proceedings article.
Part of the work was done while both authors were in Snowbird.
We are also grateful to M.~Shoemaker for stimulating discussions.
The research of H.~F. is supported by SwissMAP and grant ERC-2012-AdG-320368-MCSK in the group of Rahul Pandharipande at ETH Z\"urich. The research of Y.P.~L. was partially supported by the NSF.

\section{Quantum Lefschetz in generating functions and finite graphs}

The goal of this section is to reformulate some results in \cite{FL2} in terms of generating functions.
We assume the readers are familiar with \cite{FL2}.
Whenever possible, we follow the notations in \cite{FL2} and \cites{CGT, LPbook}.

\subsection{Recursions in Givental's quantization formalism}
Recall the setup in \cite{FL2}.
Let $D\subset X$ be a smooth hypersurface in a smooth projective variety $X$ and let 
\[
\fX=Bl_{D\times \{0\}}X\times P^1,
\]
be the ``masterspace'' associated to $(X,D)$.
There is a birational morphism $\fX\rightarrow X\times P^1$, and it can be composed with projections to $X$ and $P^1$. 
We denote the first composition by $p:\fX\rightarrow X$ and the second by $\pi:\fX\rightarrow P^1$. 
The fiber 
\[
\fX_0 := \pi^{-1}(\{0\}) \cong X \cup_D P_{D}(\sO\oplus\sO(D))
\] 
is the union of $X$ and $P_{D}(\sO\oplus\sO(D))$. These two pieces glue transversally along the hypersurface $D \subset X$ and the section
\[
D \cong P_{D}(\sO(D))\subset P_{D}(\sO\oplus\sO(D)).
\]
As in \cite{FL2}, we use the following notations
\begin{itemize}
	\item $X_\infty :=\pi^{-1}(\{\infty\}) \cong X$;
	\item $X_0$ is the irreducible component of $\fX_0$ which is isomorphic to $X$;
	\item $D_0 :=P_D(\sO)\subset P_{D}(\sO\oplus\sO(D)) \subset \fX_0$.
\end{itemize}

Endow $P^1$ with the $\C^*$ action such that it has weight $-1$ on $T_0 P^1$. 
Since $\fX$ is the blow-up of a fixed locus, \emph{there is an induced $\C^*$ action on $\fX$}. 
This action acts trivially on $X$, but scales the fibers of the $P^1$ fibration $P_{D}(\sO\oplus\sO(D))$.
Under this $\C^*$ action, the fixed loci are
\[
 \text{(a)} \, D_0, \quad \text{(b)} \, X_0, \quad \text{(c)} \, X_{\infty} .
\]
Also as in \cite{FL2}, when the $\C^*$ acts on the fibers on $E$ by ``positive'' scaling, sending a vector $v$ to $\lambda v$, we write $E^+$ for the equivariant bundle and use the notation
$\langle \dotsb \rangle_{g,n,\beta}^{X,E^+}$
for the corresponding equivariant twisted invariant.
Similarly, $\langle \dotsb \rangle_{g,n,\beta}^{X,E^-}$ stands for invariants with inverse scaling action $ v\mapsto \lambda^{-1} v$. 
In general, if these distinctions are not called for, $\langle \dotsb \rangle_{g,n,\beta}^{X,tw}$ will stand for general twisted invariants.

Let $\{ e_{\mu} \}$ be a basis of $H(X)$. 
The equivariant total descendant potential of $X$ is denoted by
\[
 \mathcal{D}_{eq}^{X} (t) = \exp \left(\sum_{g=0}^{\infty} \hbar^{g-1} F_g^{X} (t) \right)
 \]
 \text{where}
 \begin{equation}
     F_g^{X} (t) = \sum_{\beta} \sum_{n=0}^{\infty} \frac{q^{\beta} \langle t, \ldots, t \rangle_{g, n, \beta}^{\fX}}{n!}, \quad 
 t := \sum_{k, \mu} t^{\mu}_k e_{\mu} \psi^k. 
 \end{equation}
We note that the summation over $n$ includes $n=0$.
By the results of \cites{aGhg, aGqh, CGT} we have
\begin{equation} \label{e:1}
 \mathcal{D}_{eq}^{\fX} = \widehat{M}^{\fX} \left( \mathcal{D}^D \mathcal{D}^{X_0} \mathcal{D}^{{X}_{\infty}} \right),
\end{equation}
where $\widehat{M}$ is an operator coming from quantization of quadratic hamiltonians \cites{aGqh, ypLnotes}, and similar to the one in \cite{CGT}*{Theorem~1.4}.
Note that here our one-dimensional orbits of the $\mathbb{C}^*$ action are not isolated, but they do not depend on any parameter, i.e, constant in families, much like the situation in \cite{CGT}.
Therefore the formula there holds without any modification.
We don't really need to know the exact formula for $\widehat{M}^{\fX}$.
What we need is that $\widehat{M}^{\fX}$ is a product of a few operators, the most important of which is $\widehat{R}$
\begin{equation} \label{e:R}
 \widehat{R} = \exp \left( \frac{\hbar}{2} \sum_{\mu_1, \mu_2} \sum_{k_1, k_2} \Delta^{\mu_1 \mu_2}_{k_1 k_2} \frac{\partial}{\partial t^{\mu_1}_{k_1}}  \frac{\partial}{\partial t^{\mu_2}_{k_2}} \right),
\end{equation}
where $\Delta$ can be obtained from certain generating function of genus zero invariants, or equivalently the fundamental solution of the Dubrovin connection.
When $\widehat{R}$ acts on the product of $\mathcal{D} =\exp { F}$, it can be organized by \emph{hypergraphs} where $\Delta$ accounts for the edge factor and $\p^{\mu_1}_{k_1} \p^{\mu_2}_{k_2}$ connects two (not necessarily distinct) \emph{hypervertices}  on which two differentiations act.
See Equation (17) in \cite{CGT} and Part II of \cite{LPbook} for more details.

Implicit in the Equation~\eqref{e:1} is the dependence on $\op{NE}(\fX)$ and $H(\fX)$.
There is a natural map 
\[
(\iota_{D_0})_*: \op{NE} (D) \to \op{NE}(\mathfrak{X}),
\] 
where $\iota_{D_0}: D_0 \to \fX$ is the embedding of the fixed component.
Let $H(D)_X$ be the subgroup of $H(D)$, pulled back from $H(X)$.
There is a natural lifting 
\[
H(D)_X  \to H(\fX), 
\]
by pulling back the class in $X$ to $\fX$.
Denote their images by $\op{NE}(\fX)_D$ and $H(\fX)_D$ respectively.
Let $\{ T_{\mu} \}$ now stand for a basis of $H(D)_X$.
By abusing the notation, the same symbols $\beta_D$ and $t_D := \sum_{\mu, k} t^{\mu}_k T_{\mu} \psi^k$ will be used for the classes and their images under the above maps.
We now \emph{restrict the variables of GW theory on $\fX$ from $\op{NE}(\fX)$ to $\op{NE}(D)$ (for Novikov variables $q^{\beta_D}$) and $H(\fX)$ to $H(D)$ (for $t_D$)}. Let $\mathcal{D}^{X_\infty, \so^+}$ be the total descendant potential of $X_\infty$ twisted by $\so$ with fiberwise $\C^*$ action of weight $1$. 

\begin{proposition} \label{p:1}
With the restriction of $t=t_D$ and $q^{\beta} = q^{\beta_D}$, we have
\begin{equation} \label{e:2}
  \mathcal{D}_{eq}^{\fX} = \widehat{M} \left( \mathcal{D}^D \mathcal{D}^{X_0} \right) \mathcal{D}^{{X}_{\infty}, \so^+} ,
\end{equation}
where $\widehat{M}$ acts only on the first two $\mathcal{D}$'s, and the last factor is the descendant potential twisted by the Euler class of the trivial bundle.
\end{proposition}

\begin{proof}
This follows from \eqref{e:1}. 
The only difference is the single factor of $\mathcal{D}^{X_\infty, \so^+}$, on which the operator $\widehat{M}$ does not act.
As in \cite{FL2}*{\S~2.3}, the restriction of the curve classes implies that there will be no curve connecting $X_{\infty}$ with $\fX_0$.
The contribution from $X_{\infty}$ therefore is a single factor by itself as twisted invariants.
We also note that, due to the quantum Riemann--Roch theorem in \cite{CG}, there is an (explicit) operator $\widehat{\Gamma}^{X_\infty}$ such that 
\begin{equation} \label{e:qrr}
 \mathcal{D}^{X_{\infty}, \so^+} = \left( \widehat{\Gamma}^{X_{\infty}} \right)^{-1} \mathcal{D}^{X_{\infty}},
 \end{equation}
where $\Gamma$ depends on the twisting data.
\eqref{e:qrr} similarly applies to $D$ and $X_0$.
$\left( \widehat{\Gamma}^{X_{\infty}} \right)^{-1}$ was part of the ingredient in $\widehat{M}^{\fX}$ and the only one which concerns $X_{\infty}$. Removing $\left( \widehat{\Gamma}^{X_{\infty}} \right)^{-1}$ from $\widehat{M}^{\fX}$ we get $\widehat{M}$ acting only on the first two factors. The process also takes $\mathcal{D}^{X_{\infty}}$ back to $\mathcal{D}^{{X}_{\infty}, \so^+}$, as the localization formula dictates.
\end{proof}

The recursion in \cite{FL2} can be phrased in this context as follows.
Let 
\[ 
 I:= \langle \Omega \rangle^{D}_{g, n, \beta_D}, \quad \Omega= \prod_i T_{\mu_i} \psi_i^{k_i}
\]
 be the GW invariant in question on $D$.
That is, $I$ is the coefficient of $q^{\beta_D} \prod_i t^{\mu_i}_{k_i}$ in $F_g^D (t_D)$, or equialently, the coefficient of $\hbar^{g-1}q^{\beta_D} \prod_i t^{\mu_i}_{k_i}$ in $\log \mathcal{D}^{D}$.
As explained in \cite{FL2}*{Lemma~2.14} that, due to the difference of the virtual dimensions
\begin{equation} \label{e:3}
  \op{vdim}^{\fX}_{g,n,\beta_D} - \op{vdim}^{D}_{g,n,\beta_D} = 2 -2g + (\beta_D, D)^X,
\end{equation}
the corresponding GW invariant $\langle \Omega \rangle^{\fX}_{g, n, \beta_D} =0$ whenever 
\begin{equation} \label{e:ic}
 2 -2g + (\beta_D, D)^X > 0
\end{equation}
Therefore, the coefficient of $\hbar^{g-1} q^{\beta_D} \prod_i t^{\mu_i}_{k_i}$ of $\log \mathcal{D}_{eq}^{\fX}$ vanish.

The vanishing gives an equation on the right hand side of \eqref{e:2}.
Treating all (twisted) invariants on $X$ as known, this equation has as the leading term exactly $(-1)^{1-g} \lambda^{-2 +2g - (\beta_D, D)} I$.
That is, $I$ can be written in terms of invaraints on $D$ and $\fX$ of ``lower orders'', with invariants on $X$ as coefficients. Since invariants on $\fX$ can be expressed in terms of invaraints on $D$ of equal or lower orders, this gives a recursion.
The initial conditions are decided by \eqref{e:ic} when the vanishing, or equivalently the inequlity in \eqref{e:ic}, no longer holds.

\begin{remark}[On insertions of primitive classes]\label{r:prim}
In Proposition \ref{p:1}, we restrict the parametrized insertion $t$ to $t_D$. 
When $t$ is a general class in $H(\fX)$, Equation \eqref{e:2} still holds, although in this case the degree bound \eqref{e:ic} is more complicated.

Since our primary objective is to obtain the GW invariants for $D$, the insertion classes missing from the restriction $t=t_D$ are the primitive classes $H_{prim}(D)$ in $H^{d} (D)$, where $d = \dim_{\mathbb{C}} (D)$.
Recall $\fX$ is a blow-up of $X\times P^1$. Denote the exceptional divisor by $E$ and the inclusion $\iota_E : E\hookrightarrow \fX$. $E$ is a projective bundle over $D$ and admits a projection $\pi:E\rightarrow D$.
By the blowup formula of cohomology, $H^{d+2} (\fX)$ is generated by the images of pushforwards from $H^d(E)$ and pullbacks from $H^{d+2} (X \times P^1)$.
A convenient choice for $H_{prim}(D)$ in \eqref{e:2} can be obtained as follows.

Given a collection of $\omega_i\in H_{prim}(D), 1\leq i\leq m$, consider 
\[\tilde\omega_i = (\iota_E)_*\pi^*\omega_i.\]
It is easy to see that $\tilde\omega_i|_{D_0}=-\lambda \omega_i$. Similar to the above analysis,
\begin{equation} \label{e:1d}
\langle \tilde\omega_1,\ldots,\tilde\omega_m, t_D, \ldots , t_D \rangle_{g,n+m,\beta_D}^\fX = (-\lambda)^m\langle \omega_1 , \ldots , \omega_m, t_D,\ldots,t_D \rangle_{g,n+m,\beta_D}^D + \ldots .
\end{equation}
This leads to a modified degree condition:
The left hand side of \eqref{e:1d} vanishes if
\begin{equation} \label{e:ic'}
 2 -2g + (\beta_D, D)^X > m.
\end{equation}
As before, the vanishing of the LHS produces a relation.
Despite the worse degree condition, it does produce a way to compute GW invariants of $D$ with primitive insertions, whereas the methods in \cite{CGL, GJR2} so far have not.
These relations can be combined with the monodromy invariance requirement of GW invariants. 
More detailed study is in progress in a separate project. 
\end{remark}

\subsection{Calabi--Yau threefolds as the divisor $D$}
Simplifications occur in case $D$ is a Calabi--Yau threefold, when the $n$-pointed Gromov--Witten invariants can be reconstructed from $n=0$ GW invariants (with no insertion).
The generating function for $n=0$ invariants in genus $g$, called genus $g$ partition function, is a function of the Novikov variable $q$ and is independent of $t$, the insertion variables.
See \cite{FL2}*{\S~1.3} for two closed formulae relating $n$-pointed generating functions to the partition function.
Consequently, $\mathcal{D}^D (t_D)$ is completely determined by $\mathcal{D}^D (t_D =0)$, for which we will determine a simplified form of recursion relation.

For simplicity, we spell out the case when $D \subset X$ is the quintic threefold in $P^4$, although all discussions extend to general CY3 without difficulty, albeit with more complicated notations.

Since $t_D =0$, $F^D_g$ depends only on $q^{\beta_D}$. 
Due to \eqref{e:3} and the fact $\op{vdim}^D =0$, $F^{\fX}_g (t_D =0, q)$ is a polynomial in $q$ of degree ${ \leq \displaystyle \frac{2g-2}{5}}$, denoted by $P_g (q)$.
Since the information of $P_g(q)$ is equivalent to our initial conditions in e.g., \cite{FL2}*{Theorem~1.1} in the case of CY3, we can assume that $\mathcal{D}_{eq}^{\fX}$ is the initial data. 
Hence one can rephrase Proposition~\ref{p:1} as follows.
\begin{corollary} \label{c:1}
\[
  \left[ \widehat{M(q)} \left( \mathcal{D}^D \mathcal{D}^{X_0} \right) \right]_{t_D=0}  
  = e^{\sum_{g=0}^{\infty} {\hbar}^{g-1} P_g(q) } \left( \mathcal{D}^{{X}^{\infty}, \so^+} \right)^{-1}.
\]
Note that in the above equation, $P_g$ and $ \mathcal{D}^{X}$ are the initial data and the equation completely determines $\mathcal{D}^D$ from the initial data.
\end{corollary}

\begin{proof}
By \eqref{e:qrr} the generating functions of twisted GW invariants are determined by those of the untwisted GW invariants.
The coefficient of $q^d$ of the logarithm of the LHS of the equation in Corollary~\ref{c:1} has a single highest order term for $D$ as $\langle \cdot \rangle_{g,0,d}$ with nonzero constant coefficient. 
\end{proof}

\begin{remark}
We remark that the operator $\widehat{M(q)}$ involves differentiation with respect to $t$, denoted $\widehat{R}$ in \eqref{e:R}.
One can only set $t_D =0$ after the differentiation. Hence one has to use $n \neq 0$ pointed GW invariants for $D$. However, as already explained, they are all explicitly determined by $0$-points ones.

Secondly, one can formally rewrite \eqref{e:2} as
\[
 \mathcal{D}^D  = \left( \mathcal{D}^{X_0} \right)^{-1} \widehat{M}^{-1}  \left( \mathcal{D}_{eq}^{\fX} \left( \mathcal{D}^{{X}^{\infty}, \so^+} \right)^{-1} \right) ,
\]
When one restrict to CY3 case with $t_D=0$, $ \mathcal{D}_{eq}^{\fX} $ will apparently need to use more than $e^{\sum_{g=0}^{\infty} {\hbar}^{g-1} P_g(q)}$ for the input due to extra insertions created by the differentiation in $ \widehat{M}$. 
However, the differentiation in  $\widehat{M}$ carries degree in $q$, and the apparently extra information is actually in the lower order and by induction was already determined by \eqref{e:2}.
\end{remark}

\subsection{Remarks on hypergraphs}
In a series of papers \cite{aGhg, aGqh}, A.~Givental devised a summation over \emph{hypergraphs} via the localization techniques.
Initially spelled out for the case when the torus fixed points are isolated, this has been extended to more general cases in, e.g., \cite{CGT}, which we will follow.
See also \cites{ypLnotes, LPbook}.

According to Givental's organizations of localization in terms of contributions from hypergraphs, there are only a \emph{finite} number of hypergraphs for each genus.
The above relation can be written in terms of finite hypergraphs.
These hypergraphs are the decorated graphs of the dual graphs of strata in $\Mbar_{g,0}$.
The decoration assigns to each hypervertex a connected component of fixed loci, to which torus fixed curve maps.
The hyperedges are created by $\hat{R}$ acting on a product of $\mathcal{D}^D$ and $\mathcal{D}^{X}$.
The fact that $\hat{R}$ is of the form of exponential of quadratic derivatives in \eqref{e:R} and $\mathcal{D}^Y = \exp F^Y$ creates hypergraphs via Wick's formula. Here we use $Y$ to refer to a general variety independent of the context $(X,D)$.
The finiteness of the hypergraph in each power of $\hbar$, accounting for genus, has to do with the fact that the domain curves are connected.
We very briefly sketch the graphs. The interested readers can consult \cites{aGhg, aGqh, CGT, LPbook}.

In the equivariant setting, as we are now, one can interpret the hypergraphs in terms of the fixed point localization.
For simplicity, assume $g \ge2$. 
The fixed point loci can be labeled by graphs, with vertices as either nodes or a connected union of irreducible compoents mapping to a connected component of the fixed loci $Y^T$.
An edge connects two (not necessarily distinct) vertices if they intersect. 
(Contributions of) fixed loci of $\Mbar_{g,n} (Y, \beta)$ can be indexed by their images under $\op{st} : \Mbar_{g,n} (Y, \beta) \to \Mbar_{g,0}$.
We call a vertex \emph{a hypervertex} if it stablizes to an irreducible component of the stable $(g,0)$ curve.
The (necessarily genus zero) trees of vertices and edges which map to the intersecting points of hypervertices are called \emph{hyperedges}. \footnote{Here our terminology is different from that of \cites{aGhg, aGqh, CGT} and is consistent with \cite{LPbook}. What we call vertices and edges here are called joints and legs, while hypervertices and hyperedges are called vertices and edges there.}
A hypergraph is a connected graph consisting of hypervertices and hyperedges, and with decoration $(g_i, f_i)$ on the the $i$-th hypervertice, where $f_i$ is the connected component of fixed loci $Y^T$ associated to the sub-curve and $g_i$ is the genus of its normalization.

It is easy to see that there are only finitely many hypergraphs for each fixed genus $g$.
There are only finitely many dual stable graphs for $\Mbar_{g,0}$ and the possibility of decoration is finite.
For example, in our case, there are three (single-vertex) hypergraphs with no hyperedge, corresponding to three connected components of fixed loci $D_0, X_0, X_{\infty}$.
When the number of hyperedges is nonzero, the hypervertices must be associated to either $D_0$ or $X_0$, due to the constraint on the curve class.

Hence the recursion relation in the previous subsections can also be phrased as one associated only to finitely many graphs for each genus, as in \cites{GJR, GJR2}, or in the work of Chang, Guo, Li in \cites{CGL, CGL2}.

\begin{remark}
We note that Givental's quantization can be generalized to semisimple case without equivariant theory.
This is the Givental--Teleman classification of semisimple theories.
See \cites{aGhg, cT}.
\end{remark}

\section{Modifications of the masterspace}
In this section, we describe a few modifications of our original masterspace $\fX$ and consequently different recursion relations.
Even though it is unclear whether these new recursions are easier to ``solve'', they bring other types of (twisted, orbifold) invariants into play and may be of interest on their own.

\subsection{Hypersurface and other twisted theories}
We relate hypersurface invariants with invariants of $X$ twisted by other line bundles.

Choose a line bundle $L$ over $X$. Instead of blowing up $X\times \Pp^1$, we consider a $\Pp^1$-bundle $\mathscr P_L=\Pp_{X}(\sO\oplus L))$. There are two sections $(\mathscr P_L)_0$ and $(\mathscr P_L)_\infty$ both of which are isomorphic to $X$, with normal bundles $L$ and $L^{\vee}$, respectively. Denote the hypersurface in $(\mathscr P_L)_0$ by $D_0$. Define
\[
\mathfrak X(L)=\Bl_{D_0} P_L
\]
There is a $\C^*$ action on $P_L$, and it induces an action on $\fX(L)$. As before, denote the strict transforms of $(P_L)_0$ and $(P_L)_\infty$ by $X_0$, $X_\infty$, respectively. There is another connected component of fixed locus in the exceptional locus, isomorphic to $D$. By a slight abuse of notation, denote it as $D_0$. To sum it up, the fixed loci under the induced $\C^*$ action are again,
\[
 \text{(a)} \, D_0, \quad \text{(b)} \, X_0, \quad \text{(c)} \, X_{\infty}.
\]
The difference comes from their normal bundles. One identifies that
\begin{enumerate}
 \item $N_{D_0/\fX(L)}\cong L\oplus \sO_D(D)\otimes L^\vee$, 
 \item $N_{X_0/\fX(L)}\cong \sO_D(-D)\otimes L$, 
 \item $N_{X_\infty/\fX(L)}\cong L^\vee$.
\end{enumerate}
When $L$ is a trivial line bundle, we recover our previous construction. 

Consider a curve class $\beta\in \NE(\fX(L)))$ lying in the subgroup $\NE(D_0)\hookrightarrow \NE(\fX(L))$. Gromov--Witten invariants of genus $g$ with curve class $\beta$ can be recursively determined by lower terms if the following is satisfied.
\begin{equation}\label{eqn:vd}
\text{vdim}(\Mbar_{g,n}(\fX(L),\beta))>\text{vdim}(\Mbar_{g,n}(D_0,\beta)).
\end{equation}
This is explained in \cite[Section 2.3]{FL2}. Briefly, for a set of insertions (restrictions from $X$) whose degrees add up to $\text{vdim}(\Mbar_{g,n}(D_0,\beta))$, we can always find a lift in $\fX$. If we integrate the lift against $[\Mbar_{g,n}(\fX(L),\beta)]^{\text{vir}}$, the integral vanishes due to dimensional reason. Equation \eqref{eqn:vd} happens when
\begin{equation}\label{eqn:vdim}
    \displaystyle\int_\beta \left(c_1(L) + c_1(\sO_D(D)\otimes L^\vee)\right) = \displaystyle\int_\beta c_1(D) > 2(g-1),
\end{equation}
where $\beta$ is regarded as a curve class on $D_0$. Thus, we conclude that the degree condition does not change under the modification. However, the recursion relations now uses different twisted invariants of $X$.

\begin{example}
When $X=\Pp^4$, $D$ is a smooth quintic hypersurface and $L=\sO(1)$, we conclude that quintic invariants with degree $d>(2g-2)/5$ can be recursively determined by invariants of $\Pp^4$ twisted by $\sO(-4)$ and by $\sO(-1)$, plus some lower degree quintic invariants, whereas the ``original'' master space has the quintic invariants related to invariants of $\Pp^4$ twisted $\sO$ and $\sO(-5)$.
\end{example}

\subsection{Quintic $3$-fold and Hurwitz--Hodge integrals}
The construction in this subsection, as it stands now,  only applies to the case when there exists a crepant contraction $\pi: \fX \to \overline{\fX}$, such that $X_0 \subset \fX$ is contracted to a point with quotient singularity. That is, there is a $K$-equivalence, or crepant transformation, between $\fX$ and an orbifold crepant resolution $\tilde{\overline{\fX}} \to \overline{\fX}$.
Abusing the notation, as it is standard in the literature, we identify $\tilde{\overline{\fX}}$ with $\overline{\fX}$.
This is the case for the quintic $3$-folds inside $\Pp^4$, as we shall explain. 
Further developments are in an ongoing project of the first author with Longting Wu.

The construction may start with the deformation to the normal cone $\fX$ when $X=\Pp^4$ and $D=Q$ the quintic hypersurface. In this case, $X_0\cong \Pp^4$ and $N_{X_0/\fX}=\sO_{\Pp^4}(-5)$. Therefore, locally (in fact, Zariski-locally) $X_0$ is a local $\Pp^4$. Let $U\subset \fX$ be the Zariski neighborhood of $X_0$ that is isomorphic to the total space of $\sO(-5)$. It is well-known that $\sO(-5)$ is a crepant resolution of the smooth Deligne--Mumford stack $[\C^5/\mu_5]$. In particular,
\[
U \setminus X_0\cong [\C^5/\mu_5] \setminus [\text{pt}/\mu_5].
\]
As a result, there is a contraction 
\[
 \pi: \fX\rightarrow \overline{\fX},
\]
where $X_0$ is contracted to $[\text{pt}/\mu_5]$. More concretely, what we do here is simply replacing $U$ by $[\C^5/\mu_5]$.

Alternatively, one can describe $\overline{\fX}$ as follows. 
Consider the weighted projective space $\Pp(5,1,1,1,1,1)$. $\Pp^4$ embeds into it in the following way. 
\[
\Pp^4\cong \Pp(1,1,1,1,1)\subset \Pp(5,1,1,1,1,1).
\]
One can check the normal bundle of this $\Pp^4$ is $\sO(5)$. Consider the quintic hypersurface $Q\subset \Pp^4$. The normal bundle of $Q$ inside $\Pp(5,1,1,1,1,1)$ is now $\sO(5)\oplus\sO(5)$. We claim that
\[
\overline{\fX}\cong\Bl_{Q}\Pp(5,1,1,1,1,1).
\]
The only $[\text{pt}/\mu_5]$ corresponds to the stacky point in $\overline{\fX}$, and the strict transform of $\Pp^4$ corresponds to $X_\infty$. The exceptional divisor is isomorphic to $Q\times \Pp^1$, and intersect with $X_\infty$ at a quintic threefold. Such $Q\times \Pp^1$ also appears in $\fX$ as the strip connecting $Q_0$ and $X_\infty$. Thus, the other end that does not intersect $X_\infty$ corresponds to $Q_0$ of $\overline{\fX}$. A rigorous identification is left to the readers.

\begin{remark}
In fact, it is possible to realize both $\fX$ and $\overline{\fX}$ as invariant hypersurfaces of GIT quotients under certain action of $(\C^*)^3$ on $\C^9$. The crepant contraction $\fX \to \overline{\fX}$ is a consequence of wall-crossing of variations of stability conditions. 
The GIT phase of $\overline{\fX}$ occurred in a stimulating discussion in our different project with M.~Shoemaker.
\end{remark}

On $\overline{\fX}$, there is a $\C^*$ action inherited from $\fX$. The fixed loci are $Q_0, [\text{pt}/\mu_5]$ and $X_\infty$. Note that under the birational contraction of $\fX$, a neighborhood of $Q_0$ is left untouched. Therefore, we can similarly choose a curve class $\beta_Q\in \NE(Q_0)\subset\NE(\overline{\fX})$ and apply localization. Let $d$ be the degree of $\beta_Q$ as a curve class in $P^4$. Under the same bound $d>(2g-2)/5$, invariants of $Q$ can be recursively determined by Gromov--Witten invariants of $[\C^5/\mu_5]$ ($\C^*$ acting with weights $(1,1,1,1,1)$). The whole process is completely parallel to \cite{FL2}, and the invariants of $X_\infty$ will vanish by virtual-dimensional reason.

\section{New relations for the ambient space $X$}
So far, the technique has been used to find relations for the hypersurfaces, assuming the data from the ambient space is known.
Now we turn the idea on its head and show that our construction can also be used to find relations for the ambient spaces under certain favorable conditions. 
These favorable conditions are designed so that the GW invariants of $D$ will not enter the recursion relations.
We show two examples in this section.

Both examples use the same set-up as before. Consider the deformation to the normal cone of a smooth pair $(X,D)$ and the same $\C^*$-action has fixed loci $D_0, X_0, X_\infty$. 
However, the curve classes $\beta$ no longer comes from $D_0$.
Our strategy is to choose the conditions, which might involve the geometry of $(X,D)$ or the numerical conditions on the curve class $\beta$ such that the GW invariants on $D_0$ will not appear in recursion relation.

\subsection{Adding fiber classes} \label{s:3.1}

Let $f\in \NE(\fX)$ be the curve class corresponding to a fiber on the exceptional divisor between $D_0$ and $X_0$. Choose a curve class $\beta_X\in \NE(X_0)\subset \NE(\fX)$. Suppose $\alpha_1,\ldots,\alpha_n\in H^*(\fX)$ are insertions. Since we only care about relations on $X_0$, localization on $\Mbar_{g,n}(\fX,\beta_X+kf)$ yields a relation when
\[
\sum\limits_{i=1}^n \deg(\alpha_i) < (1-g)(\dim(X)-2) + \displaystyle\int_{\beta_X} c_1(\sO(-D)) + k + n.
\]
Since we do not wish to involve invariants on $D_0$, certain conditions have to be satisfied. For example, it will work if $\beta_X+kf$ is not numerically equivalent to a connected curve with an irreducible component in $D_0$. (If $(X,D) = (P^4, \text{quintic})$, this is the case when $k \leq 4$.)
We note that this is a very strong condition which limits the applicability. 
%

\begin{example}
We try to give a complete and self-contained statement in the following special case. Let $X$ be a smooth variety, and $D$ be a divisor such that $(\gamma,D)>1$ for
any nonzero $\gamma\in NE(X)$. Fix a nonzero $\beta_X\in NE(X)$ such
that \[(\beta_X,c_1(TX)-D)+1>0.\] Let $N=dim(X)$. In twisted Gromov--Witten theory, we give a weight-$1$ fiberwise $\C^*$ actions on the twisting bundles. Let $\lambda$ be the equivariant parameter.
\begin{proposition}
\begin{enumerate}
\item[(1)] ($D$ is not necessarily effective) Assume the above conditions, we have
\[
\langle \dfrac{1}{\lambda-D-\psi} \rangle_{1,1,\beta_X}^{X,\sO(-D)} = - \dfrac{1}{24\lambda^2}\langle \dfrac{D\cdot c_{N-1}(TX)}{\lambda-D-\psi} \rangle^{X,\sO(-D)}_{0,1,\beta_X} - \dfrac{1}{24\lambda}\langle \dfrac{c_{N-1}(TX)}{\lambda-D-\psi} \rangle^{X,\sO(-D)}_{0,1,\beta_X}.
\]
\item[(2)] Let $D$ be a smooth effective divisor, assume the above conditions, we have
\[
\langle \dfrac{D}{\lambda-D-\psi} \rangle_{1,1,\beta_X}^{X,\sO(-D)} = - \dfrac{1}{24\lambda}\langle \dfrac{ D\cdot c_{N-1}(TX) }{\lambda-D-\psi} \rangle^{X,\sO(-D)}_{0,1,\beta_X} - \dfrac{1}{24}\langle \dfrac{ D\cdot c_{N-2}(TX) }{\lambda-D-\psi} \rangle^{X,\sO(-D)}_{0,1,\beta_X}.
\]
\end{enumerate}
\end{proposition}
Relation (2) is derived in the above setup using $k=1$. We note that relation (1) is, however, derived by directly localizing on $\Pp_X(\sO\oplus \sO(D))$ and hence not directly related to our discussion.
We include it as a comparison to (2).
\end{example}

\subsection{Hypersurfaces with trivial Gromov--Witten theory}
Let $X$ be a Fano $3$-fold whose $-K_X$ is base-point free. If we choose $D$ to be a smooth member of $|-K_X|$, it is a K3-surface by the adjunction formula, and hence has vanishing Gromov--Witten invariants except certain degree $0$ invariants. In principle, relations on $X$ can be produced if the choice $\beta= \beta_X+kf\in \NE(\fX)$ ($\beta_X, f$ as in Section~\ref{s:3.1}) and insertions $\alpha_1,\ldots,\alpha_n\in H^*(\fX)$ satisfy
\[
\sum\limits_{i=1}^n \deg(\alpha_i) < (1-g) + k + n.
\]
Since it is independent on $\gamma$ and $k$ can be arbitrarily large, this is a not a very serious constraint. 
For example, the following sample relation holds for all GW invariants of Fano threefolds whose aniticanonical divisors are base-point free.

\begin{proposition}
If $X$ is a Fano $3$-fold with base-point free anticanonical divisor and $\beta_X\in \NE(X)$ is a nonzero curve class, the following relation holds without any extra conditions.
\[
\langle \dfrac{-K_X}{\lambda+K_X-\psi} \rangle_{1,1,\beta_X}^{X,\sO(K_X)} = - \dfrac{1}{24\lambda}\langle \dfrac{ -K_X\cdot c_{2}(TX) }{\lambda+K_X-\psi} \rangle^{X,\sO(K_X)}_{0,1,\beta_X} - \dfrac{1}{24}\langle \dfrac{ K_X^2 }{\lambda+K_X-\psi} \rangle^{X,\sO(K_X)}_{0,1,\beta_X}.
\]
\end{proposition}

There are likely more equations of this type for any Fano threefolds. 
We wish to be able to report some interesting applications of these equations in the future.

\bibliographystyle{amsxport}
\bibliography{ref}

\end{document}